\documentclass[a4paper]{amsart}

\usepackage{amssymb}

\usepackage[alphabetic]{amsrefs}

\usepackage{a4wide,enumerate,graphicx,amsthm}

\usepackage{hyperref}

\usepackage{amscd}

\newtheorem{Satz}{Theorem}[section]
\newtheorem{Proposition}[Satz]{Proposition}
\newtheorem{Lemma}[Satz]{Lemma}
\newtheorem{Corollary}[Satz]{Corollary}

\newtheorem{Vermutung}[Satz]{Conjecture}

\theoremstyle{definition}
\newtheorem{Definition}[Satz]{Definition}
\newtheorem{Bemerkung}[Satz]{Remark}

\newtheorem{Beispiel}[Satz]{Example}

\newcommand{\Z}{\mathbb{Z}}
\newcommand{\R}{\mathbb{R}}
\newcommand{\Q}{\mathbb{Q}}
\newcommand{\p}{\partial}
\newcommand{\ti}{\widetilde}
\newcommand{\ov}{\overline}
\newcommand{\e}{\varepsilon}

\DeclareMathOperator{\vol}{vol}
\DeclareMathOperator{\sys}{sys}

\DeclareMathOperator{\dist}{dist}
\DeclareMathOperator{\wid}{width}
\DeclareMathOperator{\len}{length}
\DeclareMathOperator{\fil}{FillRad}

\begin{document}

\title{Macroscopic Band Width Inequalities}

\author{Daniel R\"ade}

\address{Universit\"at Augsburg, Institut f\"ur Mathematik, 86135 Augsburg, Germany}

\email{daniel.raede@math.uni-augsburg.de}

\begin{abstract}
Inspired by Gromov's work on \emph{Metric inequalities with scalar curvature} we establish band width inequalities for Riemannian bands of the form $(V=M\times[0,1],g)$, where $M^{n-1}$ is a closed manifold. We introduce a new class of orientable manifolds we call \emph{filling enlargeable} and prove:
If $M$ is filling enlargeable and all unit balls in the universal cover of $(V,g)$ have volume less than a constant $\frac{1}{2}\varepsilon_n$, then $\wid(V,g)\leq1$.
We show that if a closed orientable manifold is enlargeable or aspherical then it is filling enlargeable. Furthermore we establish that whether a closed orientable manifold is filling enlargeable or not only depends on the image of the fundamental class under the classifying map of the universal cover.
\end{abstract}

\keywords{Large manifolds, volumes of balls, systolic geometry, band width inequality}

\subjclass[2010]{53C23}

\maketitle

\section{Introduction}

This paper is concerned with a class of manifolds called \emph{bands}. While Gromov gives a very general definition in \cite{Gro18}*{Section 2}, the following is enough for our purposes:

\begin{Definition}
Let $M^{n-1}$ be a closed smooth manifold. We call $V=M\times[0,1]$ the \emph{band} over $M$. If $g$ is a smooth Riemannian metric on $V$, we call $(V,g)$ a \emph{Riemannian band} and define the \emph{width} of the band, denoted $\wid(V,g)$, to be the distance between $M\times\{0\}$ and $M\times\{1\}$ with respect to $g$.
\end{Definition}

In \cite{Gro18}*{Section 2} Gromov estimates the width of certain classes of Riemannian bands from above under the assumption that the scalar curvature $Sc(g)$ of the metric is bounded from below by a positive constant. He calls an orientable band $V^{n}$ \emph{over-torical} if it admits a continuous map to $T^{n-1}\times[0,1]$ with non zero degree, which maps $M\times\{0\}$ to $T^{n-1}\times\{0\}$ and $M\times\{1\}$ to $T^{n-1}\times\{1\}$. He  proves:

\begin{Satz}[Gromov] Let $V^n$ be an over-torical band. If $n\leq 8$ and $g$ is a Riemannian metric on $V$ with $Sc(g)\geq\sigma>0$, then 
\begin{equation}
\wid(V,g)\leq 2\pi\sqrt{\frac{n-1}{\sigma n}}.
\end{equation}\qed
\end{Satz}

In general he conjectures \cite{Gro18}*{11.12, Conjecture C} that if a closed manifold $M$ of dimension $\geq 5$ does not admit a metric of positive scalar curvature, then (1) holds for Riemannian bands $(V,g)$, diffeomorphic to $M\times [0,1]$, with $Sc(g)\geq\sigma>0$.\\

Using the Dirac operator, Zeidler \cite{RZ19}*{Theorem 1.4} and Cecchini \cite{Ce20}*{Theorem D} proved this conjecture for closed spin manifolds $M$ with non-vanishing Rosenberg index (an obstruction to admitting a metric of positive scalar curvature). See also \cite{RZ20} for more information regarding their methods.\\

In this paper we establish band width inequalities for a Riemannian band $(V^n,g)$ under a different condition on the metric $g$. Instead of a lower scalar curvature bound we assume that all metric balls of a certain radius $R>0$ in the universal cover $(\ti{V},\ti{g})$, have volume less than $\varepsilon_nR^n$ for some small constant $\varepsilon_n>0$, only depending on the dimension $n$.\\

This type of metric condition was studied by Larry Guth \cites{Gu10a,Gu10b,Gu11,Gu17}. We quickly summarize the motivation behind it, before we state our main results.

\subsection{Macroscopic scalar curvature}

The value $Sc(g,p)$ of the scalar curvature at a point $p$ in a Riemannian manifold $(M^n,g)$ appears as a coefficient in the Taylor expansion of the volume of a geodesic ball of radius $R$ around $p$:
\begin{equation} 
{\rm vol}(B_R(p))= \omega_nR^n\left(1-\frac{Sc(g,p)}{6(n+2)}R^2+\mathcal{O}(R^4)\right),
\end{equation}
where $\omega_n$ is the volume of the unit ball in euclidean space $\mathbb{R}^n$. It follows that if the scalar curvature of $(M,g)$ at a point $p$ is positive, there is a $\lambda(M,g,p)>0$ such that all $R$-balls in $(M,g)$ centered at $p$ with $R<\lambda(M,g,p)$ have $\vol(B_R(p))<\omega_nR^n$.\\
Hence, for $R$ small enough, the scalar curvature of $(M,g)$ at $p$ can be quantified by comparing the volumes of $R$-balls around $p$ with their counterparts in $S^n$, $\R^n$ and $\mathbb{H}^n$, the standard simply connected manifolds with constant scalar curvature.\\
If we carry out this volume comparison (see \cite[Section 7]{Gu10a}) for all radii $R>0$ we get:

\begin{Definition}
Let $(M^n,g)$ be a Riemannian manifold and $p\in M$. The \emph{macroscopic scalar curvature} at scale $R$ at $p$, denoted by $Sc_R(p)$, is defined to be the number $S$ such that the volume of the ball of radius $R$ around any lift of $p$ in the universal cover of $M$ equals the volume of the ball of radius $R$ in a simply connected space with constant curvature and with scalar curvature $S$. 
\end{Definition}

The universal cover is used to ensure that the macroscopic scalar curvature of a flat torus is zero at any scale. If $M$ is closed and one does not consider balls in the universal cover, but in $(M,g)$ itself, then the macroscopic scalar curvature would be positive at a large enough scale. For infinitesimally small radii we have $\lim_{R\rightarrow0}Sc_R(g,p)=Sc(g,p)$ by (2).\\

The most prominent conjecture in this vein is due to Gromov \cite{Gro85}:

\begin{Vermutung} \label{conj:gr}
Let $g$ be a metric on a closed aspherical manifold $M^n$. For any radius $R>0$ there is a point $p$ in the universal cover $(\ti{M},\ti{g})$ with $\vol(B_R(p))\geq\omega_nR^n$.
\end{Vermutung}

\begin{Bemerkung}
This would imply that a closed aspherical manifold does not admit a metric with positive macroscopic scalar curvature at any scale $R>0$. It follows that  it can not admit a metric with positive scalar curvature either, since:
\begin{itemize}
\item For a \emph{closed} Riemannian manifold $(M^n,g)$ with positive scalar curvature, we can find a \emph{uniform} constant $\lambda(M,g)$ such that all $R$-balls in $(M,g)$ with $R<\lambda(M,g)$ have $\vol(B_R(p))<\omega_nR^n$ (compare the paragraph after (2)).\\
\item For $R$ small enough the $R$-balls in $(M,g)$ agree with the $R$-balls in $(\ti{M},\ti{g})$.
\end{itemize}
\end{Bemerkung}

In \cite[Corollary 3]{Gu11} Guth proved a version of Conjecture \ref{conj:gr}, where $\omega_n$ is replaced by a smaller constant $\varepsilon_n$. Recently his results were further generalized and improved upon by Papasoglu \cite[Theorem 3.1]{Pap19} and Nabutovsky \cite[Theorem 2.7]{Nab19}.

\subsection{Main results}

We introduce a class of orientable manifolds we call \emph{filling enlargeable} (see Definiton \ref{def:fill}). This notion combines Gromov and Lawson's \cite{Gro80} classical definition of \emph{enlargeability} with the \emph{filling radius} of a complete oriented Riemannian manifold, a metric invariant that was introduced by Gromov in \cite{Gro83}*{Section 1}.\\

Some important features of filling enlargeable manifolds we establish in this paper are:
\begin{itemize}
\item If a closed orientable manifold is enlargeable or aspherical, then it is filling enlargeable (see Propositions \ref{prop:enlfill} and \ref{prop:asphfill}).
\item Closed filling enlargeable manifolds are essential. In fact we prove an even stronger statement in Theorem \ref{thm:3} (see also Remark \ref{rem:zess} and Remark  \ref{rem:qess}).
\item For all $n\geq 1$ there is a constant $\e_n>0$ such that the following holds. Let $M^n$ be a closed filling enlargeable manifold and $g$ a Riemannian metric on $M$.  For any radius $R>0$ there is a point $p$ in the universal cover $(\ti{M},\ti{g})$ with $\vol(B_R(p))\geq\e_nR^n$ (see Proposition \ref{prop:wenl}).
\end{itemize}

In light of the third point and in the context of macroscopic scalar curvature we prove the following macroscopic analog of the band width inequalities with scalar curvature for bands over filling enlargeable manifolds:

\begin{Satz} \label{thm:2}
For all $n\geq1$ there is a constant $\varepsilon_n>0$ such that the following holds. Let $M^{n-1}$ be a closed filling enlargeable manifold and $V:=M\times[0,1]$. If $g$ is a Riemannian metric on $V$ with the property that all unit balls in the universal cover $(\ti{V},\ti{g})$ have volume less than $\frac{1}{2}\varepsilon_n$, then $\wid(V,g)\leq 1$.
\end{Satz}

It is a natural question to ask whether Theorem \ref{thm:2} holds true for all essential manifolds. However, as it turns out, there are some immediate counterexamples (see Example \ref{ex}). In order to obtain similar results one has to add further assumptions regarding the systole (the length of the shortest noncontractible loop) of $(V,g)$:

\begin{Satz} \label{thm:main}
For all $n\geq1$ there is a constant $\e_n>0$ such that the following holds. Let $M^{n-1}$ be a closed essential manifold and $V:=M\times[0,1]$. Let $g$ be a Riemannian metric on $V$ and $0<R<\frac{1}{2}\sys(V,g)$. If every ball of radius $R$ in $(\ti{V},\ti{g})$ has volume less than $\frac{1}{2}\e_nR^n$, then $\wid(V,g)\leq R$.
\end{Satz}

In Section 4 we follow ideas of Brunnbauer and Hanke \cite{BH10} and study some functorial properties of filling enlargeable manifolds. In particular we construct a vector subspace $H_n^{sm}(B\Gamma;\Q)\subset H_n(B\Gamma;\Q)$ of 'small classes' in the rational group homology of a finitely generated group $\Gamma$ and prove:

\begin{Satz} \label{thm:3}
Let $M$ be a closed oriented manifold of dimension $n$. Then $M$ is filling enlargeable if and only if $\phi_*[M]\in H_n(B\pi_1(M);\Q)$ is not contained in $H_n^{sm}(B\Gamma;\Q)$.
\end{Satz}

\textbf{Acknowledgments:} This work is part of my ongoing doctoral dissertation project at Augsburg University. I want to thank my advisor Bernhard Hanke for his continued support and Benedikt Hunger for helpful comments after reading an earlier version of this article. The author was supported by a doctoral grant from the German Academic Scholarship Foundation.

\section{Largeness properties of manifolds}

\subsection{Width, filling radius and systole} Even though our results are stated for smooth manifolds, almost all of our actual work takes place in the setting of simplicial complexes with piecewise smooth metrics (see \cite{bab02}*{Section 2}). 

\begin{Definition}
By a Riemannian metric on a $k$-simplex $\Delta^k$ we understand the pullback of an arbitrary Riemannian metric on $\R^k$ via an affine linear embedding $\Delta^k\hookrightarrow \R^k$. A \emph{Riemannian metric} on a simplicial complex $X$ is given by a Riemannian metric $g_\tau$ on every simplex, such that $g_{\tau'}\equiv g_{\tau}|_{\tau'}$ for $\tau'\subset\tau$. We call $(X,g)$ a \emph{Riemannian polyhedron}.
\end{Definition}

A Riemannian metric $g$ enables us to measure the lengths of piecewise smooth curves in a simplicial complex $X^n$. As for Riemannian manifolds one obtains a path metric $d_g$ on $X^n$ if $X^n$ is connected. Moreover there is an obvious notion of $n$-dimensional Riemannian volume, coinciding with the $n$-dimensional Hausdorff measure. In the following we introduce some classical metric invariants used to describe the size of $(X,g)$.\\
We remind the reader that a metric space is called \emph{proper} if every closed and bounded subset is compact. Furthermore a continuous map between metric spaces is called \emph{proper} if the preimage of every bounded set is bounded. It is a classical result, that a path metric space is proper if and only if it is locally compact and complete.

\begin{Bemerkung}\label{rem:proper}
In most of the literature a continuous map between topological spaces is called proper, if the preimage of every compact set is compact. Our definition coincides with the classical one for continuous maps between proper metric spaces.\\
While Theorem \ref{thm:2} and Theorem \ref{thm:main} are concerned with compact metric spaces, the relevance of proper metric spaces is that we will be talking about possibly infinite-sheeted covering spaces of these compact spaces.
\end{Bemerkung}

\begin{Definition} \label{def:wid}
The $k$-dimensional Alexandrov width $UR_{k}(X,g)$ of a proper Riemannian polyhedron $(X,g)$ is the infimum over all values $R>0$, such that: there is a continuous map $f:X\rightarrow Y^{k}$ to a locally finite $k$-dimensional simplicial complex $Y$, with the property that the preimage $f^{-1}(\tau)$ of every simplex $\tau\subset Y$ is contained in a metric ball of radius $R$ in $X$.
\end{Definition}

While originally appearing in the context of topological dimension theory, this notion was popularized in Riemannian Geometry by works of Gromov (see \cite{Gro83}*{Appendix 1} or \cite{Gro88}). In the course of this paper we will often use an equivalent definition of the $k$-dimensional Alexandrov width in terms of open covers, namely: $UR_{k}(X,g)$ is the infimum over all $R$ such that $X$ admits a locally finite cover by open sets of radii $\leq R$ (i.e. contained in a metric ball of radius $R$) and multiplicity $\leq k+1$. The following Lemma, which (only in the compact case) appears in \cite{Gro88}*{Section (H)}, modulates between this definition and Definition \ref{def:wid}. For the convenience of the reader, and since we need to be a little bit more careful, when considering non-compact polyhedra, we include a proof of this result.

\begin{Lemma} \label{lem:mod}
A proper Riemannian polyhedron $(X,g)$ has $UR_{k}(X,g)<R$ if and only if there is a locally finite open cover of $X$ with multiplicity $\leq k+1$ and radius $<R$
\end{Lemma}

\begin{proof}
Suppose that $X$ has a locally finite open cover $O_i$ with multiplicity $\leq k+1$ and radius $<R$. If we denote by $N$ the nerve of this cover, the nerve map $\phi$ associated to a partition of unity, subordinate to $O_i$, is a continuous map $X\rightarrow N$ and for each simplex $\tau \subset Y$ the preimage $\phi^{-1}(\tau)$ is contained in one of the sets $O_i$.\\
For the other direction suppose that $(X,g)$ has $UR_{k}(X,g)< R$ i.\ e.\ there is a locally finite simplicial complex $Y^k$ and a continuous map $f:X\rightarrow Y$ such that for each simplex $\tau \subset Y$ the preimage  $f^{-1}(\tau)$ is contained in a ball of radius $UR_{k}(X,g)+\delta<R$ for some small $\delta>0$. We equip $Y$ with the canonical path metric i.e. we equip every simplex with the standard euclidean metric and consider the induced path metric on $Y$. First we want to show that the map $f$ is in fact proper. A bounded subset $K\subset Y$ intersects only finitely many simplices of $Y$. In particular it is contained in a finite subcomplex $K'\subset Y$. Since $(X,g)$ is proper and the preimage of every simplex is closed and bounded we conclude that $f^{-1}(K')$ is bounded. But then $f^{-1}(K)$ is a closed subset of a bounded set and thus it is bounded as well.\\
Since $Y$ is $k$-dimensional, we prove the following: there is an open cover $O_i$ of $Y$ with multiplicity $\leq k+1$ such that $f^{-1}(O_i)$ is contained in a ball of radius $<R$ for all $i$. We start by claiming that for every simplex $\tau\subset Y$ there is a constant $\e(\tau)>0$ such that $f^{-1}(U_{\e(\tau)}(\tau))$ is contained in a ball of radius $<R$ in $(X,g)$. By assumption $f^{-1}(\tau)$ is contained in a ball of radius $UR_{k}(X,g)+\delta<R$ around a point $p\in X$. Now we assume for a contradiction, that for every $\ell\in\mathbb{N}$ there is a point $y_\ell\in U_{1/\ell}(\tau)$ with $f^{-1}(y_\ell)\not\subset B_{R-1/\ell}(p)$ and choose a preimage $x_\ell\notin B_{R-1/\ell}(p)$. Up to a subsequence $(y_\ell)_{\ell\in\mathbb{N}}$ converges to a point $y\in\tau$. Since $(x_\ell)_{\ell\in\mathbb{N}}$ is contained in the compact set $f^{-1}(\overline{U_1(\tau)})$ there is a subsequence converging to a point $x\in X\backslash B_R(p)$, which contradicts the continuity of $f$.\\
To construct $O_i$ we consider the skeleta of $Y$ one at a time, starting with the vertices. If $v\in Y$ is a vertex then $f^{-1}(B_{\e(v)}(v))$ is contained in a ball of radius $<R$ in $(X,g)$ and we can assume that for two vertices $v_1$ and $v_2$ the balls $B_{\e(v_1)}(v_1)$ and $B_{\e(v_2)}(v_2)$ are disjoint. Now let $\gamma\subset Y$ be an edge connecting vertices $v_1$ and $v_2$. Denote $\gamma'=\gamma\backslash(B_{\e(v_1)}(v_1)\cup B_{\e(v_2)}(v_2))$. Then $f^{-1}(U_{\e(\gamma)}(\gamma'))$ is contained in a ball of radius $<R$ in $(X,g)$ and by possibly making $\e(\gamma)$ smaller we can arrange that $U_{\e(\gamma)}(\gamma')$ does not intersect $U_{\e(\eta)}(\eta')$ for any other edge $\eta\subset Y$. We continue this process for all higher skeleta of $Y$ and make sure that for each skeleton all newly introduced open sets are disjoint, which provides the upper bound on the multiplicativity of this cover. Finally $f^{-1}(O_i)$ is the required cover of $(X,g)$.
\end{proof}

For us the most important result concerning the Alexandrov width of a proper Riemannian polyhedron $(X^n,g)$ is the following theorem (which is a version of  \cite{Pap19}*{Theorem 3.1}), providing an upper bound on $UR_{n-1}(X,g)$ under the assumption that for a fixed radius $R>0$ all $R$-balls in $(X,g)$ have very small volume. 

\begin{Satz}\label{thm:pap}
For all $n\geq1$ there is a constant $\varepsilon_n>0$ such that the following holds: If $(X^n,g)$ is a proper Riemannian polyhedron and $R>0$ is a radius such that for every $x\in X$ the volume of the ball $B_R(x)$ is bounded from above by $\varepsilon_nR^n$. Then $UR_{n-1}(X,g)\leq R$.\qed
\end{Satz}

\begin{Bemerkung}
In the case of complete smooth Riemannian manifolds this was first proved by Guth in \cite{Gu17}. His theorem was generalized by Liokumovich, Lishak, Nabutovsky and Rotman \cite{LLNR} to metric spaces and Hausdorff content instead of volume.\\
Building on ideas from \cite{Gu10b} the proof of this result was significantly simplified by Papasoglu \cite{Pap19}*{Theorem 3.1}. Finally Nabutovsky \cite{Nab19}*{Theorem 2.6} was able to improve the constant $\varepsilon_n$ from an exponential bound to a linear one in the case of compact Riemmanian polyhedra.
\end{Bemerkung}

\begin{Bemerkung}
Any smooth Riemannian manifold $(M,g)$ becomes a Riemannian polyhedron by choosing a smooth triangulation of $M$. For the rest of this paper, a Riemannian metric on $M$, if not explixitly required to be smooth, is a polyhedral metric with respect to some smooth triangulation of $M$. Furthermore all manifolds are assumed to be connected.
\end{Bemerkung}

Next, we revisit the \emph{filling radius} of a complete oriented Riemannian manifold $(M^n,g)$ (we mostly follow \cite{BH10}*{Section 2}). The orientation corresponds to a fundamental class $[M]\in H_n^{lf}(M;\Z)$ in locally finite homology (i.e. we consider infinite chains with the property that each bounded subset intersects only finitely many singular simplices).\\
Denote by $L^\infty(M)$ the vector space of all functions $M\rightarrow\R$ with the uniform 'norm' $\|-\|_\infty$. We consider the affine subspace $L^\infty(M)_b$ of $L^\infty{M}$, that is parallel to the Banach space of all bounded functions on $M$ and contains the distance function $d_g(x,-)$ for some $x\in M$. Notice that $\|-\|_\infty$ defines an actual metric on $L^\infty(M)_b$. The \emph{Kuratowski embedding} 
\begin{displaymath}
\iota_g: (M,d_g)\hookrightarrow L^\infty(M)_b \ \ \ x\mapsto d_g(x,-)
\end{displaymath}
is isometric by the triangle inequality.

\begin{Definition}
The \emph{filling radius} of $(M,g)$ is defined as 
$$ \fil(M,g):=\inf{\{r>0|\iota_{g*}[M]=0\in H_n^{lf} (U_r(\iota_gM);\mathbb{Q})\}},$$
where $U_r(\iota_gM)$ denotes the open $r$-neighborhood of $\iota_gM$ in $L^\infty(M)_b$. If the set on the right hand side is empty we say that $(M,g)$ has infinite filling radius.
\end{Definition}

If $M$ is closed, then $L^\infty(M)_b$ is the vector space of all bounded functions and the above definition coincides with the classical one from \cite{Gro83}*{Section 1}. We remind the reader that for an arbitrary metric space $S$, the space $L^\infty(S)$ of all functions has the following universal property.

\begin{Lemma} \label{lem:univ}
If $Y\subset X$ is a subspace of a metric space and if $f:Y\rightarrow L^\infty(S)$ is an $L$-Lipschitz map, then there exists an extension $F:X\rightarrow L^\infty(S)$ which is also $L$-Lipschitz.\qed
\end{Lemma}

This is due to Gromov \cite{Gro83}*{Page 8}. There it is only stated for closed Riemannian manifolds, but the proof works in the general setting. The extension $F$ is given by
\begin{displaymath}
F_x(v):=\inf_{y\in Y}(f_y(v)+L\cdot d(x,y)).
\end{displaymath}
For our purposes we also need the following property of $F$:

\begin{Lemma} \label{lem:proper}
If $f$ is proper and $d(\cdot,Y)$ is uniformly bounded in $X$, then $F$ is proper.
\end{Lemma}

\begin{proof}
Let $K\subset L^\infty(S)$ be a bounded subset. Let $C>0$ be a uniform upper bound of $d(\cdot,Y)$ in $X$. Since $F$ is $L$-Lipschitz we have $F_x\in U_{LC}({\rm im}(f))$ for all $x\in X$. It follows that if $K$ does not intersect $U_{LC}({\rm im}(f))$, then $F^{-1}(K)=\emptyset$.\\
Hence assume that $K\subset U_{LC}({\rm im}(f))$. Consider the bounded set $A:=U_{LC}(K)\cap {\rm im}(f)$. Since $f$ is proper, $f^{-1}(A)\subset Y$ is bounded as well. We claim that $F^{-1}(K)\subset U_C(f^{-1}(A)).$\\
Let $x\notin U_C(f^{-1}(A))$ be arbitrary. Since $d(\cdot,Y)$ is uniformly bounded in $X$, there is a $y\in Y$ with $d(x,y)<C$. Consequently $y\notin f^{-1}(A)$. But then
$$d(F_x,F_y)=d(F_x,f_y)\leq L\cdot d(x,y)<LC$$
and thus $F_x\notin K$ because otherwise $f_y\in A$.
\end{proof}

Let $(N^n,h)$ be another complete oriented Riemannian manifold with fundamental class $[N]\in H_n^{lf}(N;\Z)$. The mapping degree of a continuous map $f:M\rightarrow N$ is well defined if $f$ is proper or $N$ is closed and $f$ is almost proper (i.e. constant outside a compact set).\\
A fundamental property of the filling radius, which follows from Lemma \ref{lem:univ} and Lemma \ref{lem:proper} and will be used throughout this paper, is that it behaves well under distance decreasing maps of non-zero degree.

\begin{Lemma} \label{lem:fill}
If $f:(M,g)\rightarrow (N,h)$ be a distance decreasing proper (or almost proper) map with ${\rm deg}(f)\neq0$, then $\fil(M,g)\geq\fil(N,h)$.
\end{Lemma}

\begin{proof}
Assume that for some $R>0$ the fundamental class $\iota_{g*}[M]$ bounds a locally finite chain in $U_R(\iota_gM)\subset L^\infty(M)_b$. By Lemma \ref{lem:univ} the distance decreasing map $\iota_{g'}\circ f:M\rightarrow L^\infty(N)_b$ extends to a distance decreasing map $F:L^\infty(M)_b\rightarrow L^\infty(N)_b$, which maps $U_R(\iota_gM)$ to $U_R(\iota_{h}N)$. Since $\iota_{g'}\circ f$ is proper, $F$, when restricted to $U_R(\iota_gM)$, is proper as well (see Lemma \ref{lem:proper}) i.\ e. preimages of bounded sets are bounded. By construction, $F$ maps $\iota_{g*}[M]\in H_n^{lf}(U_R(\iota_gM);\Q)$ to ${\rm deg}(f) \iota_{h*}[N]\in H_n^{lf}(U_R(\iota_{h}N);\Q)$. Hence $\iota_{h*}[N]$ vanishes in $U_R(\iota_{h}N)\subset L^\infty(N)_b$
\end{proof}

\begin{Bemerkung}
Here it is important that we chose rational coefficients for our definition of the filling radius. For integral coefficients we would need to restrict to the case ${\rm deg}(f)=1$ in the Lemma above. 
\end{Bemerkung}

The next lemma regarding the relationship between the filling radius and the Alexandrov width is taken from \cite{Gro83}*{Appendix 1, Example in (B) and (D)}.

\begin{Lemma} \label{lem:ineq}
For a complete Riemannian manifold $(M^n,g)$ we have $\fil(M,g)\leq UR_{n-1}(M,g)$.\qed
\end{Lemma}
 
Finally Theorem \ref{thm:main} is related to systolic geometry, the study of the following metric invariant:

\begin{Definition}
The \emph{systole} $\sys(X,g)$ of a Riemannian polyhedron $(X,g)$ is defined to be the infimum of all lengths of noncontractible closed piecewise smooth curves in $X$.
\end{Definition}

\begin{Definition} \label{def:ess}
A connected finite $n$-dimensional simplicial complex $X$ is called \emph{essential} if the classifying map $f:X\rightarrow K(\pi_1(X),1)$ induces a homomorphism $$f_*:H_n(X;G)\rightarrow H_n(K(\pi_1(X),1);G)$$ with non-trivial image for coefficients $G=\Z$ or $\Z_2$. A closed manifold $M$ is called \emph{essential} if any smooth triangulation of $M$ produces an essential simplicial complex.
\end{Definition}

It is a central result in systolic geometry \cite[Appendix 2, (B1')]{Gro83} that for any compact essential Riemannian polyhedron $(X^n,g)$ the following, so called, \emph{isosystolic inequality} holds true:
\begin{equation}
\sys(X,g)\leq C(n)\vol(X,g)^{\frac{1}{n}},
\end{equation}
where $C(n)$ is a constant that only depends on the dimension $n$. For the special case of closed essential manifolds see \cite[Theorem 0.1.A]{Gro83}.

\subsection{Filling enlargeable manifolds}

With all invariants in play we now want to explain the notion of a \emph{filling enlargeable} manifold featured in Theorem \ref{thm:2}. It is based on the concept of an \emph{enlargeable} manifold, introduced by Gromov and Lawson \cite{Gro80}. 

\begin{Definition} \label{def:enl}
A closed orientable manifold $M^n$ is called \emph{enlargeable}, if for every Riemannian metric $g$ on $M$ and every $r>0$ there is a Riemannian covering $\overline{M}_r$ of $(M,g)$ and a distance decreasing almost proper (i.e constant outside of a compact set) map $f_r:\overline{M}_r\rightarrow S^n(r)$ to the round sphere of radius $r$ with ${\rm deg}(f_r)\neq 0$.
\end{Definition}

We consider the following notion, which combines the ideas of Definition \ref{def:enl} and the filling radius of a complete oriented Riemannian manifold:

\begin{Definition} \label{def:fill}
A closed orientable manifold $M^n$ is called \emph{filling enlargeable}, if for every Riemannian metric $g$ on $M$ and every $r>0$ there is a Riemannian covering $\overline{M}_r$ of $M$ with $\fil(\overline{M}_r,\overline{g})\geq r$, where $\overline{g}$ denotes the lifted metric.
\end{Definition}

\begin{Bemerkung} \label{rem:bilip}
To check whether or not a closed orientable manifold is (filling) enlargeable it suffices to consider one fixed metric $g$, since all Riemannian metrics on a closed manifold are in bi-Lipschitz correspondence. Furthermore we want to stress the fact that the covering spaces $\ov{M}_r$ might very well be infinite-sheeted i.\ e. non-compact (see Remark \ref{rem:proper}).
\end{Bemerkung}

We begin our discussion of this new class of large manifolds by proving that it contains all closed enlargeable and all orientable closed aspherical manifolds.

\begin{Proposition} \label{prop:enlfill}
If $(M,g)$ is a closed enlargeable manifold, then $(M,g)$ is filling enlargeable.
\end{Proposition}

\begin{proof}
M. Katz \cite{Kat83} proved that the filling radius of the unit sphere $S^n(1)$ is $\frac{1}{2}{\rm arccos}(-\frac{1}{n+1})$. Let $r>0$ be arbitrary and denote $r'=\frac{r}{\fil(S^n(1))}$. Since $M$ is enlargeable there is a Riemannian covering $\overline{M}_{r'}$ of $M$ and a distance decreasing almost proper map $f_r:\overline{M}_{r'}\rightarrow S^n(r')$ of nonzero degree. By Lemma \ref{lem:fill} it follows that $\fil(\overline{M}_{r'},\ov{g})\geq\fil(S^n(r'))=r$.
\end{proof}

To prove the same for orientable closed aspherical manifolds we need some more preparation.

\begin{Definition}
A proper Riemannian polyhedron $(X,g)$ is said to be \emph{uniformly contractible} if there exists a function $C:[0,\infty)\rightarrow \R_{\geq0}$ such that every ball of radius $R$ in $X$ is contractible within the ball with the same center and radius $C(R)$.
\end{Definition}

\begin{Lemma} \label{lem:geom}
Let $(X,g)$ be a contractible proper Riemannian polyhedron and assume there exists a subgroup $G$ of the isometry group that acts cocompactly on $X$. Then $(X,g)$ is uniformly contractible.
\end{Lemma}

\begin{proof}
Let  $\pi:X\rightarrow X/G$ be the quotient projection and for any $x\in X$ we consider the open ball $B_1(x)$. The projection is an open map: If $U\subset X$ is open, then $$\pi^{-1}(\pi(U))=\{x\in X:\pi(x)\in \pi(U)\}=\bigcup_{g\in G}gU$$ is the union of open sets. Hence the image $\pi(B_1(x))$ is an open neighborhood of $\pi(x)\in X/G$. Since $X/G$ is compact there are finitely many points $x_1,\ldots, x_k$ in $X$ such that $X/G\subset \bigcup_{i=1}^k \pi(B_1(x_i))$. We define $K:=\bigcup_{i=1}^k B_1(x_i)$. Then for any $x\in X$ there is a $\alpha \in G$ such that $\alpha x\in K$. Furthermore $\overline{K}$ is compact.\\
Let $R>0$ be arbitrary. Since $X$ is contractible there is a homotopy $H:X\times [0,1]\rightarrow X$ connecting $id_X$ to the constant map $X\mapsto p$ where $p\in X$ is a basepoint. Without loss of generality we assume $p\in K$. Now $H(\overline{B}_{R+{\rm diam}(K)}(p)\times [0,1])$ is a compact subset of $X$ containing $p$. As compact sets are bounded $H(\overline{B}_{R+{\rm diam}(K)}(p)\times [0,1])$ is contained in a ball of radius $C(R)-{\rm diam}(K)$ around $p$ for some $R\leq C(R)<\infty$.\\
Now let $x\in X$ be arbitrary and $\alpha\in G$ such that $\alpha x\in K$. Then $\alpha(B_R(x))\subset B_{R+{\rm diam}(K)}(p)$ and thus $B_R(x)$ is contractible within $B_{C(R)-{\rm diam}(K)}(\alpha^{-1}p)\subset B_{C(R)}(x)$.
\end{proof}

\begin{Lemma} [\cite{Gro85}]\label{lem:inf}
Let $(X^n,g)$ be a complete oriented Riemannian manifold. If $X$ is uniformly contractible then $\fil(X,g)=\infty$.\qed
\end{Lemma}

\begin{Proposition} \label{prop:asphfill}
If $(M,g)$ is an orientable closed aspherical manifold, then $(M,g)$ is filling enlargeable.
\end{Proposition}

\begin{proof}
The universal cover of a closed orientable aspherical manifold is uniformly contractible by Lemma \ref{lem:geom} and, if we fix an orientation, has infinite filling radius by Lemma \ref{lem:inf}. Hence closed oriented aspherical manifolds are filling enlargeable.
\end{proof}

We remind the reader that the product of two enlargeable manifolds is enlargeable \cite{Gro80}*{Introduction}. As, later on, our proof of Theorem \ref{thm:2} will rely on a doubling procedure, where we glue two bands $M\times[0,1]$ and obtain a copy of $M\times S^1$, we investigate whether or not the same holds true for filling enlargeable manifolds.

\begin{Lemma} \label{lem:fillprod}
Let $(M^n,g)$ be a complete oriented Riemannian manifold and $(S^1,g_r)$ the standard round circle of radius $r$. If $\fil(S^1,g_r)\geq\fil(M,g)$, then $$\fil(M\times S^1, g\oplus g_r)\geq\fil(M,g).$$
\end{Lemma}

\begin{proof}
We assume for a contradiction that $\fil(M\times S^1, g\oplus g_r)<\fil(M,g)$. Let $\e>0$ be such that $\fil(M\times S^1, g\oplus g_r)<\e<\fil(M,g)$. Using Lemma \ref{lem:univ} we extend the projections $$p_1:M\times S^1\rightarrow M \text{ and } p_2:M\times S^1\rightarrow S^1$$ to some nonexpanding maps $$P_1:L^\infty(M\times S^1)_b\rightarrow L^\infty(M)_b \text{ and } P_2:L^\infty(M\times S^1)_b\rightarrow L^\infty(S^1).$$ By our choice of $\e$ the class $\iota_{g_r*}[S^1]$ does not vanish in the $\e$-neighborhood of $\iota_{g_r}(S^1)$ in $L^\infty(S^1)$. Since we are working with field coefficients the universal coefficient theorem tells us that there is a cohomology class $[\alpha]\in H^1(U_\e(\iota_{g_r}(S^1));\Q)$ dual to $\iota_{g_r*}[S^1]$ that extends the fundamental cohomology class of $S^1$.\\
We pull back $[\alpha]$ via $P_2$ to get a cohomology class $P_2^*[\alpha]\in H^1(U_\e(\iota_{g\oplus g_r}(M\times S^1));\Q)$. Denote by $\sigma$ the locally finite fundamental cycle $\iota_{g\oplus g_r*}[M\times S^1]$. By choice of $\e$ we know that $\sigma$ bounds a chain in its $\e$-neighbourhood. Consequently the cap product $$\sigma\cap P_2^*\alpha \in C^{lf}_n(U_\e(\iota_{g\oplus g_r}(M\times S^1));\Q)$$ is a boundary as well. By construction $\sigma\cap P_2^*\alpha$ represents the same locally finite homology class as $\iota_{g\oplus g_r}[M\times\{*\}]$. We conclude that $0=\iota_{g\oplus g_r*}[M\times\{*\}]\in H^{lf}_n(U_\e(\iota_{g\oplus g_r}(M\times S^1));\Q)$.\\
Finally $P_1$ maps $U_\e(\iota_{g\oplus g_r}(M\times S^1))$ to $U_\e(\iota_{g}(M))$ and is proper when restricted accordingly. Thus $$0=P_{1*}(\iota_{g\oplus g_r*}[M\times\{*\}])=\iota_{g*}[M]\in H^{lf}_n(U_\e(\iota_{g}(M));\Q).$$ This contradicts our assumption that $\e<\fil(M,g)$.
\end{proof}

\begin{Proposition} \label{prop:fillprod}
Let $M$ be a closed filling enlargeable manifold. Then $M\times S^1$ is filling enlargeable as well.
\end{Proposition}

\begin{proof}
By Remark \ref{rem:bilip} it is enough to consider the case that $M\times S^1$ is equipped with a product metric $g\oplus g_1$, where $g$ is a Riemannian metric on $M$ and $g_1$ is the standard metric on $S^1$ with radius $1$. Let $r>0$ be arbitrary.\\
On the one hand, since $M$ is filling enlargeable, there is a Riemannian covering $(\ov{M}_r,\ov{g})$ with $\fil(\ov{M}_r,\ov{g})\geq r$. On the other hand there is a radius $\ell$ such that $(S^1,g_l)$ is a Riemannian covering of $(S^1,g_1)$ and $\fil(S^1,g_l)\geq\fil(\ov{M}_r,\ov{g})$. Using Lemma \ref{lem:fillprod} we conclude that $\fil(\ov{M}_r\times S^1,\ov{g}\oplus g_\ell)\geq r.$ Since $(\ov{M}_r\times S^1,\ov{g}\oplus g_\ell)$ is a Riemannian covering of $(M\times S^1,g\oplus g_1)$ this proves the proposition.
\end{proof}

\begin{Bemerkung}
Lemma \ref{lem:fillprod} remains true if we replace $(S^1,g_r)$ with any other closed oriented Riemannian manifold $(N,h)$ such that $\fil(N,h)\geq\fil(M,g)$. This is, however, not enough to answer the general question whether the product of two filling enlargeable manifolds is filling enlargeable, as in the definition of filling enlargeability the Riemannian coverings are not required to be compact.
\end{Bemerkung}

\subsection{Width enlargeable manifolds} In the proof of our main Theorem \ref{thm:2}, it will be convenient to consider an even more general class of manifolds, which we obtain by replacing the filling radius with the Alexandrov width in the definiton of filling enlargeability:

\begin{Definition}
A closed manifold $M^n$ is called \emph{width enlargeable}, if for every Riemannian metric $g$ on $M$ and every $r>0$ there is a covering manifold $\overline{M}_r$ of $M$ with $UR_{n-1}(\overline{M_r},\overline{g})\geq r$, where $\overline{g}$ denotes the lifted metric.
\end{Definition}

\begin{Bemerkung} \label{rem:orient}
In contrast to Definition \ref{def:fill} we don't have to restrict ourselves to orientable manifolds in this case, since the definition of the Alexandrov width does not involve a fundamental class.
Any filling enlargeable manifold is also width enlargeable by Lemma \ref{lem:ineq}. Furthermore all closed aspherical manifolds, including the non-orientable ones, are width enlargeable (the universal cover has $UR_{n-1}=\infty$ by Lemma \ref{lem:inf} and Lemma \ref{lem:ineq}).\\
It is not clear to us whether a product result like Proposition \ref{prop:fillprod} also holds for width enlargeable manifolds. If so, our band width inequality Theorem \ref{thm:2} would be true for this even more general class of manifolds.
\end{Bemerkung}

The main property of width enlargeable manifolds we are interested in is:

\begin{Proposition} \label{prop:wenl}
For all $n\geq1$ there is a constant $\varepsilon_n>0$ such that the following holds. Let $M^n$ be a width enlargeable manifold and $g$ any Riemannian metric on $M$. Then for every $R>0$ there is a point $p$ in the universal cover $(\widetilde{M},\widetilde{g})$ such that $\vol(B_R(p))\geq\varepsilon_n R^n$.
\end{Proposition}

\begin{proof} 
Let $\e_n$ be the constant from Theorem \ref{thm:pap}. Assume for a contradiction, that there is a radius $R>0$ such that the volume of all balls of radius $R$ in $(\ti{M},\ti{g})$ is bounded from above by $\varepsilon_n R^n$. Since $M$ is width-enlargeable there is a covering $\overline{M}_R$ with $UR_{n-1}(\overline{M}_r,\overline{g})>R$. On the other hand the volume of a ball of radius $R$ centered at a point $p$ in $(\overline{M}_r,\overline{g})$ is bounded from above by the volume of the $R$-ball around any lift $\ti{p}$ of $p$ in the universal cover $(\ti{M},\ti{g})$, which is smaller than $\varepsilon_n R^n$ by assumption. Since $(\overline{M}_r,\overline{g})$ is a complete locally compact path metric space it is proper by the Hopf-Rinow Theorem and  hence $UR_{n-1}(\overline{M}_r,\overline{g})\leq R$ by Theorem \ref{thm:pap}, which is a contradiction.
\end{proof}

We spend the rest of this section proving that for closed width enlargeable manifolds the isosystolic inequality (3) holds true. Using \cite{bab92}*{Corollary 8.3} we conclude that orientable closed width enlargeable manifolds are essential.\\
This result is not necessary to prove Theorem \ref{thm:2}, but it is interesting to see how our newly introduced classes fit into the hierarchy of large manifolds. For closed orientable manifolds we get:
\begin{displaymath}
\text{Aspherical, Enlargeable }\subseteq\text{ Filling enlargeable }\subseteq\text{ Width enlargeable }\subseteq\text{ Essential }
\end{displaymath}

\begin{Lemma}\label{lem:lift}
Let $(M,g)$ be a Riemannian manifold and $(\overline{M},\overline{g})$ be a Riemannian cover. Let $B_R(p)$ be a ball with the property that the inclusion homomorphism $\pi_1(B_R(p))\rightarrow \pi_1(M)$ is trivial. Then $B_R(p)$ lifts to a collection of disjoint open sets in $\overline{M}$, each of which is the ball of radius $R$ around some lift $p'$ of $p$.
\end{Lemma}

\begin{proof}
Choose a preimage $p'$ of $p$ under the covering projection. Since $\pi_1(B_R(p))\rightarrow \pi_1(M)$ is trivial, there is a unique lift $f':B_R(p)\rightarrow (\overline{M},\overline{g})$ with $f'(p)=p'$. Let $x\in B_R(p)$ be arbitrary. There is a path $\gamma:[0,1]\rightarrow M$ with $\gamma(0)=p$, $\gamma(1)=x$ and $\len(\gamma)<R$. Now $\gamma$ lifts to a path $f'\circ\gamma$ connecting $p'$ and $f'(x)$ and $\len(f'\circ\gamma)=\len(\gamma)<R$. Thus $f'(B_R(p))$ is contained in $B_R(p')$. Let $x'\in B_R(p')$ be arbitrary. There is a path $\sigma:[0,1]\rightarrow \overline{M}$ with $\sigma(0)=p'$ and $\sigma(1)=x'$ and $\len(\sigma)<R$. But then $\pi\circ\sigma$ is a path in $M$ connecting $p$ and $\pi(x')$ with the property $\len(\pi\circ\sigma)=\len(\sigma)<R$. Thus $\pi(B_R(p'))\subset B_R(p)$. Furthermore $\pi$ is injective on $B_R(p')$. If it were not then $B_R(p)$ would contain a loop that is not contractible in $M$. Together with the fact that $\pi\circ f'$ is the identity on $B_R(p)$ this implies that both inclusions $f'(B_R(p))\subset B_R(p')$ and $\pi(B_R(p'))\subset B_R(p)$ are in fact equalities.\\
If $p''$ is a different preimage of $p$ and $f'':B_R(p)\rightarrow(\overline{M},\overline{g})$ the respective lift of $B_R(p)$ with $f''(p)=p''$ then $f(B_R(p))\cap f''(B_R(p))=\emptyset$. If not then $p'$ and $p''$ could be joined by a path within $\pi^{-1}(B_R(p))$ which would project to a noncontractible loop in $B_R(p)$.
\end{proof}

\begin{Proposition}\label{prop:lift2}
Let $(M,g)$ be a Riemannian manifold and assume that $UR_k(M,g)<\frac{1}{2}\sys(M,g)$. If $(\overline{M},\overline{g})$ is any Riemannian cover of $(M,g)$, then
\begin{displaymath}
UR_k(\overline{M},\overline{g})\leq UR_k(M,g)
\end{displaymath}
\end{Proposition}

\begin{proof}
Denote $UR_k(M,g)=r$ and let $(U_i)_{i\in\mathcal{I}}$ be an open cover of radius $\leq r+\delta$ and multiplicity $\leq k+1$, where $\delta<\frac{1}{3}(\sys(M,g)-2r)$. For a fixed $i\in\mathcal{I}$ there is a point $p_i\in M$ such that $U_i\subset B_{r+\delta}(p_i)$. We claim that the inclusion homomorphism $\pi_1(B_{r+\delta}(p_i))\rightarrow\pi_1(M)$ is trivial. Let $\gamma$ be a loop in $B_{r+\delta}(p_i)$.  There is a subdivision $0=t_0<\ldots<t_\ell=1$ of the unit interval such that $d_g(\gamma(t_i),\gamma(t_{i+1}))<\delta$. We can connect each $\gamma(t_i)$ to $p$ by a minimizing geodesic $\sigma_i$ of length $<R+\delta$ and thus $\gamma$ is homotopic to the concatenation of $\ell$ 'thin' triangles $\sigma_i(1-t)\cdot \gamma|_{[t_i,t_{i+1}]}\cdot \sigma_i$. Each of these triangles is a loop based at $p$ of length less than $\sys(X,g)$ and thus contractible, which implies that $\gamma$ is contractible as well.\\
Lemma \ref{lem:lift} the ball $B_{r+\delta}(p_i)$ lifts to a collection of disjoint $(r+\delta)$-balls around the preimages of $p_i$. Thus $U_i$ also lifts to a collection of disjoint open sets in $\overline{M}$, each of which is contained in a ball of radius $r+\delta$. If we do this for each $i\in\mathcal{I}$ we produce an open cover of $(\overline{M},\overline{g})$ with radius $\leq r+\delta$. The multiplicity of this cover is still $\leq k+1$. To see this assume for a contradiction, that a point $x\in\overline{M}$ is contained in $k+2$ open sets $U_1,\ldots,U_{k+2}$ of the cover. Then $\pi(x)$ is contained in the sets $\pi(U_1),\ldots,\pi(U_{k+2})$. But each of these sets $\pi(U_j)$ corresponds to an open set in the cover of $M$. Since only $k+1$ of these sets can contain $x$ we can assume without loss of generality that $\pi(U_1)=\pi(U_2)=U$. Let $p_1\in U_1$ and $p_2\in U_2$ be such that $\pi(p_1)=\pi(p_2)$. Since $U_1\cap U_2\neq\emptyset$, it follows that $U_1=U_2$.
\end{proof}

\begin{Corollary} \label{cor:widsys}
If a closed manifold $M^n$ is width enlargeable, then $\frac{1}{2}\sys(M,g)\leq UR_{n-1}(M,g)$ for any metric $g$ on $M$.\qed
\end{Corollary}

\begin{Proposition}
Let $M^n$ be a closed width enlargeable manifold and $g$ be a Riemannian metric on $M$. Then:
$$\sys(M,g)\leq C(n)\vol(M,g)^{1/n},$$ where $C(n)$ is a constant that only depends on the dimension $n$.
\end{Proposition}

\begin{proof}
If $M$ is width enlargeable and $g$ is a metric on $M$ then $\sys(M,g)\leq 2UR_{n-1}(M,g)$ by Corollary \ref{cor:widsys}. If we take the radius $R$ to be $\e_n^{-\frac{1}{n}}\vol(M,g)^{\frac{1}{n}}$ in Theorem \ref{thm:pap}, then the assumption automatically holds true and we conclude that $UR_{n-1}(M,g)\leq \e^{-\frac{1}{n}}\vol(M,g)^{1/n}$. Thus the isosystolic inequality holds with $C(n)=2\e_n^{-\frac{1}{n}}$.
\end{proof}

\begin{Bemerkung}\label{rem:zess}
It follows directly from \cite{bab92}*{Corollary 8.3} that an orientable closed width enlargeable manifold is essential.
\end{Bemerkung}

\section{Proof of the main results}

Let $M^{n-1}$ be a manifold and $g$ a Riemannian metric on $V:=M\times [0,1]$. The key idea in our proof of Theorem \ref{thm:2} is to construct a Riemannian polyhedron $(D,g_d)$ from a Riemannian band $(V,g)$ by taking its \emph{metric double}, which is constructed like this:\\
We fix a smooth triangulation of $V$.
Let $(V_1,g_1)$ be the Riemannian polyhedron obtained from $(V,g)$ via this triangulation. Let $V_2:=M\times[-1,0]$ and $g_2$ be the pullback metric under the diffeomorphism $s:V_2\rightarrow V_1$ $(x,y)\mapsto (x,-y)$. Since $s$ is a diffeomorphism we can also pull back the smooth triangulation from $V_1$ to $V_2$ via $s$, giving $(V_2,g_2)$ the structure of a Riemannian polyhedron.\\
In order to get $(D,g_d)$ we take the disjoint union of $(V_1,g_1)$ and $(V_2,g_2)$, and glue them together along their (isometric) boundaries i.e. $M\times\{-1\}\sim_{s} M\times\{1\}$ and $M\times\{0\}\sim_{id} M\times\{0\}$. The result is a simplicial complex $D$. Every simplex of $D$ is a proper subset of the subcomplexes $V_1$ or $V_2$ (here we identify $V_1$ and $V_2$ with their images under the quotient map from their disjoint union to $D$). Thus we can define a Riemannian metric $g_d$ on $D$ by setting $g_d|_\tau=g_i|_\tau$ for any simplex $\tau$ in $D$ depending on whether $\tau$ is a subset of $V_1$ or $V_2$.\\
To get acquainted with the notion of the double we establish the following:

\begin{Proposition} \label{prop:sys}
Let $\gamma$ be a closed noncontractible piecewise smooth curve in $(D,g_d)$. Then either $\len(\gamma)\geq2\wid(V,g)$ or there is a closed noncontractible piecewise smooth curve $\tilde{\gamma}$ in $(V,g)$ with $\len(\tilde{\gamma})=\len(\gamma)$.
\end{Proposition}

\begin{proof}
As above, consider $V_1$ and $V_2$ as subsets of $D$. Their intersection is the disjoint union of two copies of $M$, call them $M_0$ and $M_{\pm1}$. Without loss of generality we assume $\gamma(0)=\gamma(1)\in V_1$. There is a partition $0\leq t_1<\cdots<t_{2k}\leq 1$ of $[0,1]$ ($k$ might of course be 0), such that the following holds: $\gamma(t_i)\in M_0\cup M_{\pm1}$ and $\gamma([t_{2i-1},t_{2i}])\subset V_2$ while $\gamma([t_{2i},t_{2i+1}])\subset V_1$. Furthermore $\gamma([0,t_{1}])\subset V_1$ and $\gamma([t_{2k},1])\subset V_1$.\\
Of course if $t_{i}\in M_0$ and $t_{i+1}\in M_{\pm1}$ (or the other way round) for some index $i=1,\ldots,2k-1$ then $\gamma([t_{i},t_{i+1}])$ as well as $\gamma([t_{i+1},1])\cdot\gamma([0,t_{i}])$ are curves that connect $M_0$ and $M_{\pm1}$. It follows that $\len(\gamma)\geq2\cdot\wid(V,g)$.\\
Thus we assume that for all $i$ we have $\gamma(t_i)\in M_0$. Denote by $r:D\rightarrow V_1$ the map which is the identity on $V_1$ and on $V_2=M\times[-1,0]$ has the form $(x,y)\mapsto(x,-y)$ (where we consider $D$ as $M\times[-1,1]/_\sim$). This is a continuous retraction, preserving the length of curves. Let $i=1,\ldots,k$ and take the curve $c:=\gamma([t_{2i-1},t_{2i}])\subset V_2$. We claim that $c$ is homotopic to $r(c)\in V_1$ with fixed end points.\\
Let $c=(c_1,c_2)$. By assumption $c_2(t_{2i-1})=0=c_2(t_{2i})$. Furthermore $c_2(t)\in[-1,0]$ as $c\subset V_2$. Now 
\begin{displaymath}
H_1:[t_{2i-1},t_{2i}]\times[0,1]\rightarrow D \ \ (t,s)\mapsto (c_1(t),(1-s)c_2(t))
\end{displaymath}
is a homotopy between $c$ and the curve $(c_1(t),0)\subset M_0$. In the same way
\begin{displaymath}
H_2:[t_{2i-1},t_{2i}]\times[0,1]\rightarrow D \ \ (t,s)\mapsto (c_1(t),(s-1)c_2(t))
\end{displaymath}
is a homotopy between $(c_1,-c_2)=r(c)$ and $(c_1(t),0)\subset M_0$. As the endpoints are fixed in both $H_1$ and $H_2$, we can concatenate them to get a homotopy between $c$ and $r(c)$.\\
Finally, $\tilde{\gamma}:=r(\gamma)=\gamma([0,t_1])\cdot r(\gamma([t_1,t_2]))\cdot \gamma([t_2,t_3])\cdots r(\gamma([t_{2k-1},t_{2k}]))\cdot\gamma([t_{2k},1])$ is a closed curve in $V_1$ homotopic to $\gamma$ and therefore noncontractible. As $r$ preserves the length of curves, we get $\len(\tilde{\gamma})=\len(\gamma)$ and since we can view $(V_1,g_d|_{V_1})\subset (D,g_d)$ as a copy of $(V,g)$, this proves the lemma.
\end{proof}

Now we have all the necessary tools to prove our main Theorem \ref{thm:2}:

\begin{proof} [Proof of Theorem \ref{thm:2}]
Let $\varepsilon_n$ be the constant from Theorem \ref{thm:pap}. Denote by $(\ti{V},\ti{g})$ the universal cover of $(V,g)$ and let $(D,g_d)$ be the metric double of $(\ti{V},\ti{g})$. We claim that the volume of every unit ball in $(D,g_d)$ is bounded from above by $\varepsilon_n$:\\
To see this let $p_1\in \ti{V_1}$ be arbitrary and $p_2$ be the mirror image of $p_1$ in $\ti{V_2}$ (here $\ti{V_1}$ and $\ti{V_2}$ are the two copies of $\ti{V}$ used in the doubling procedure). If we denote by $q:\ti{V_1}\coprod \ti{V_2}\rightarrow D$ the quotient projection, it turns out that $B_R(q(p_1))\subset q(B_R(p_1)\cup B_R(p_2))$. In fact let $x\in B_R(q(p_1))$ i.\ e.\ there is a path $\sigma$ of length less than $R$ connecting $q(p)$ and $x$ in $D$. Now if $x\in q(\ti{V_1})$, then, using the same techniques as in the proof of Proposition \ref{prop:sys}, $\sigma$ can be modified to a path of the same length that connects $q(p_1)$ and $x$ in $q(\ti{V_1})$. Hence $x\in q(B_R(p_1))$. If, on the other hand, $x\in q(\ti{V_2})$, the $\sigma$ can be modified to a path of the same length that connects $q(p_2)$ and $x$ in $q(\ti{V_2})$.\\ 
Assume that $\frac{1}{2}\sys(D,g_d)=\wid(V,g)>1$. Let $p$ be an arbitrary point in the universal cover $(\ti{D},\ti{g}_d)$. If $\pi:\ti{D}\rightarrow D$ denotes the covering projection, then $\pi$ is injective on $B_1(p)$ as two points $p_1$ and $p_2$ in $\ti{D}$ with $p_1\neq p_2$ and $\pi(p_1)=\pi(p_2)$ have $\dist(p_1,p_2)\geq\sys(D,g_d)>2$ and any two points in $B_1(p)$ are of distance less than 2 from each other. It follows that $\pi(B_1(p))$ is isometric to $B_1(\pi(p))$ and thus $\vol(B_1(p))<\varepsilon_n$.\\
Now $(\ti{D},\ti{g}_d)$ is also the universal cover of the metric double of $(V,g)$, which, as a manifold, is just $M\times S^1$ and hence it is width enlargeable by Proposition \ref{prop:fillprod} and Lemma \ref{lem:ineq}. Furthermore $(\ti{D},\ti{g_d})$ is a complete, locally compact path metric space and thus proper by the Hopf-Rinow-Theorem. Since the volume of all unit balls in $(\ti{D},\ti{g}_d)$ is bounded from above by $\varepsilon_n$ this contradicts Proposition \ref{prop:wenl}.
\end{proof}

\begin{Bemerkung}
The same proof also works for bands over closed width-enlargeable manifolds except for the fact that we do not know whether a product result like Proposition \ref{prop:fillprod} also holds for width enlargeable manifolds i.\ e. whether $M\times S^1$ is width enlargeable if $M$ is width enlargebale (see Remark \ref{rem:orient}).\\
However, since the product of an aspherical manifold with $S^1$ is aspherical as well, Theorem \ref{thm:2} holds true for all closed aspherical manifolds, not only the orientable ones.
\end{Bemerkung}

As we stated in the introduction Theorem \ref{thm:2} does not hold true for all essential manifolds. This is due to the fact that some essential manifolds, for example $\R P^{n-1}$, actually do admit metrics with uniformly positive (macroscopic) scalar curvature at all scales. This leads to the following: 

\begin{Beispiel} \label{ex}
Let $g$ be the standard metric on $\R P^{n-1}$ induced by the round metric on the unit sphere $S^{n-1}$. Consider the direct product with an interval of arbitrary length $[0,\ell]$. We can produce a metric with any given lower bound on $Sc_1(p)$ for all $p\in \R P^{n-1}\times [0,\ell]$ by just rescaling the metric on $\R P^{n-1}$ to be very small, since the volume of a unit ball in the universal cover of the product is bounded from above by the volume of a unit ball in $S^{n-1}$, which becomes very small when rescaling the metric with a small constant.
\end{Beispiel}

We spend the rest of this section proving Theorem \ref{thm:main}. It will be crucial to relate the systoles of $(V,g)$ and $(D,g_d)$.

\begin{Lemma} \label{cor}
Let $M$ be a closed manifold and $V:=M\times[0,1]$. If $g$ is a Riemannian metric on V, then 
\begin{equation}
\sys(D,g_d)\geq {\rm min}\{\sys(V,g),2\wid(V,g)\}
\end{equation}
where $(D,g_d)$ denotes the metric double of $(V,g)$.
\end{Lemma}

\begin{proof}
As the systole is defined to be the infimum over the lengths of all closed noncontractible piecewise smooth curves, this follows immediately from Proposition \ref{prop:sys}.
\end{proof}

Lemma \ref{cor} implies that to estimate the systole or the width of a Riemannian band $(V,g)$ from above in terms of some $R>0$, it is enough to estimate the systole of the metric double $(D,g_d)$ from above in terms of $R$.\\
In order to achieve this we use Theorem \ref{thm:pap} and the next two lemmata, the second of which appears in \cite{Nab19} and is attributed to Roman Karasev.

\begin{Lemma} \label{lem:ess}
Let $M^{n-1}$ be a closed essential manifold and $V=M\times[0,1]$ a band over $M$. Then the double $D$ is essential as well.
\end{Lemma}

\begin{proof}
Since $M^{n-1}$ is closed and essential the classifying map $f:M\rightarrow K(\pi_1(M),1)$ induces a homomorphism $f_*:H_{n-1}(M;G)\rightarrow H_{n-1}(K(\pi_1(M),1);G)$ with non-trivial image for coefficients $G=\Z$ or $\Z_2$. Since $D$ is homeomorphic to  $M\times S^1$ we have $\pi_1(D)=\pi_1(M)\times\Z$. Hence a $K(\pi_1(D),1)$-space is given by $K(\pi_1(M),1)\times S^1$. The classifying map $g:D\rightarrow K(\pi_1(M),1)\times S^1$ is the direct product of $f$ and $id_{S^1}$. The general K\"unneth formula tells us that the cross product maps$$H_{n-1}(M;G)\otimes H_1(S^1;G)\rightarrow H_n(M\times S^1;G)$$ and $$H_{n-1}(K(\pi_1(M),1);G)\otimes H_1(S^1;G)\rightarrow H_n(K(\pi_1(M),1)\times S^1;G)$$ are injective and commute with the maps induced by $f, id_{S^1}$ and $g$. Let $\alpha\in H_n(M;G)$ be a class with $f_*\alpha\neq0$ and denote by $e$ the generator of $H_1(S^1;G)$. Then $g_*(\alpha\times e)=f_*\alpha\times e\neq0,$ proving the lemma.
\end{proof}

\begin{Lemma} \label{lem:kar}
If $(X^n,g)$ is an essential Riemannian polyhedron, then
$\sys(X,g)\leq 2UR_{n-1}(X,g)$.
\end{Lemma}

\begin{proof}
Assume that $\sys(X,g)>2UR_{n-1}(X,g):=2R$. Choose a covering of $X$ with multiplicity $\leq n$ by connected open sets $U_\alpha$ of radii $\leq R+\delta$ with $\delta<\frac{1}{3} (\sys(X,g)-2R)$. Let $\gamma$ be a loop contained in one of the $U_\alpha$ and $p$ be the center of a ball of radius $R+\delta$ that contains $U_\alpha$. There is a subdivision $0=t_0<\ldots<t_k=1$ of the unit interval such that $d_g(\gamma(t_i),\gamma(t_{i+1}))<\delta$. We can connect each $\gamma(t_i)$ to $p$ by a minimizing geodesic $\sigma_i$ of length $<R+\delta$ and thus $\gamma$ is homotopic to the concatenation of $k$ 'thin' triangles $\sigma_i(1-t)\cdot \gamma|_{[t_i,t_{i+1}]}\cdot \sigma_i$. Each of these triangles is a loop based at $p$ of length less than $\sys(X,g)$ and thus contractible, which implies that $\gamma$ is contractible as well. It follows that the inclusion homomorphisms $\pi_1(U_\alpha)\rightarrow \pi_1(X)$ are trivial and hence each $U_\alpha$ lifts to a collection of disjoint open sets $(\ti{U}_\alpha)_g$ (with $g\in\pi_1(X)$), homeomorphic to $U_\alpha$, in the universal cover $(\ti{X},\ti{g})$ (for more details see Lemma \ref{lem:lift}).\\
If we consider the nerves $N$ of the covering $\{U_\alpha \}$ of $X$ and $N'$ of the covering $\{(\ti{U}_\alpha)_g\}$ of $\ti{X}$, we see that $N'$ is a covering of $N$ and the following diagram commutes: $$
\begin{CD}
\ti{X}     @>>>  N'\\
@VVV        @VVV\\
X     @>>>  N.
\end{CD}$$
Here the vertical arrows are the covering projections while the horizontal arrows are the nerve maps. Let $p\in X$ be a point and $\gamma$ a noncontractible loop in $X$ based at $p$. Then $\gamma$ lifts to a path $\ti{\gamma}$ connecting two different points $\ti{p}_1$ and $\ti{p}_2$ in the fiber over $p$. By construction $\ti{p}_1$ and $\ti{p}_2$ are not contained in a common set $(\ti{U}_\alpha)$. Hence $\ti{\gamma}$ is mapped to a path connecting different points in the nerve $N'$. When projected to $N$ this yields a noncontractible loop, which agrees with the image of $\gamma$ under the nerve map $X\rightarrow N$. It follows that the induced map $\pi_1(X)\rightarrow \pi_1(N)$ is injective. Since this map is always surjective it is an isomorphism. Therefore the classifying map $X\rightarrow K(\pi_1(X),1)$ factors through the $(n-1)$-dimensional nerve $N$, so $X$ is not essential.
\end{proof}

With these ingredients we can prove Theorem \ref{thm:main}:

\begin{proof}[Proof of Theorem \ref{thm:main}]
Consider the double $(D,g_d)$ of $(V,g)$ as before. Every $R$-ball in $V$ has volume smaller than $\frac{1}{2}\e_nR^n$, and hence every $R$-ball in $D$ has volume smaller than $\e_nR^n$ (see the proof of Theorem \ref{thm:2}). Now Theorem \ref{thm:pap} implies that $UR_{n-1}(D,g_d)\leq R$. As $M$ is essential, $D$ is essential as well by Lemma \ref{lem:ess}, and hence Lemma \ref{lem:kar} implies
\begin{displaymath}
\sys(D,g_d)\leq 2UR_{n-1}(D,g_d)\leq 2R.
\end{displaymath}
Finally it follows from Lemma \ref{cor}, that ${\rm min}\{\sys(V,g),2\wid(V,g)\}\leq2R$, which proves the theorem as we assumed $\sys(V,g)>2R$.
\end{proof}

\begin{Bemerkung}
There are two ways to look at this result: on the one hand, if we assume that $2R<\sys(V,g)$ (like we did in Theorem \ref{thm:main}) the above proof produces a band width estimate.\\
On the other hand, if we replace the assumption that $2R<\sys(V,g)$ by $\wid(V,g)>R$, the above proof produces an estimate for the systole. One could consider this to be an extension of the classical systolic inequality to bands over essential manifolds which are wide enough.
\end{Bemerkung}

\section{Homological invariance of filling enlargeability}

In this section we closely follow the arguments of \cite{BH10}*{Section 3} to prove Theorem \ref{thm:3}. In order to do this, we need to extend our notion of filling enlargeability from closed oriented manifolds to rational homology classes of simplicial complexes.\\
In the following if $p:\ov{X}\rightarrow X$ is a (not necessarily connected) cover of a simplicial complex $X$ and $c\in H_n(X;\Q)$ is a (simplicial) homology class, the \emph{transfer} $p^!(c)\in H^{lf}_n(\ov{X};\Q)$ is represented by the formal sum of all possible lifts of simplices in a chain representative of $c$, where every lift of a simplex $\sigma:\Delta^n\rightarrow X$ is added with the same coefficient as $\sigma$. For more information on the transfer homomorphism (in the case of finite coverings) see for example \cite{Hat02}*{3.G}.

\begin{Definition}
A connected subcomplex $S$ of a simplicial complex $X$ is called $\pi_1$\emph{-surjective} if the inclusion induces a surjection on fundamental groups and we say that $S$ \emph{carries} a homology class $c\in H_*(X;\Q)$ if $c$ lies in in the image of the map in homology induced by the inclusion.
\end{Definition}

\begin{Definition} \label{def:fillen} (Compare \cite{BH10}*{Definition 3.1})
Let $X$ be a simplicial complex with finitely generated fundamental group and let $c\in H_n(X;\Q)$ be a (simplicial) homology class. Choose a finite connected $\pi_1$-surjective subcomplex $S\subset X$ carrying $c$. (This subcomplex exists because $\pi_1(X)$ is finitely generated.)\\
The class $c\in H_n(X;\Q)$ is called \emph{filling enlargeable}, if the following holds: For any $r>0$, there is a connected cover $p:\ov{X}\rightarrow X$ such that the class $p^!(c)\in H^{lf}_n(\ov{S};\Q)$ does not vanish in the $r$-neighborhood of the Kuratowski embedding $\iota(\ov{S})$. Here $\ov{S}=p^{-1}(S)$ (which is connected since $S$ is $\pi_1$-surjective) is equipped with the canonical path metric.
\end{Definition}

\begin{Bemerkung}
A closed oriented manifold $M^n$ is filling enlargeable (as in Definition \ref{def:fill}) if and only if its fundamental class $[M]$ is filling enlargeable (choose $S=M$ in Definition \ref{def:fillen}).
\end{Bemerkung}

As in \cite{BH10} we need to show that this definition does not depend on the choice of $S$. Let $S'\subset S$ be a smaller $\pi_1$-surjective subcomplex carrying $c$ and $r>0$ be arbitrary. Let $\ov{S}$ be a covering of $S$ such that $p^!(c)$ does not vanish in the $r$-neighborhood of $\iota(\ov{S})$. The lifted inclusion $\ov{S'}\hookrightarrow \ov{S}$ is 1-Lipschitz and extends to a nonexpanding map $L^\infty (\ov{S'})\rightarrow L^{\infty}(\ov{S})$. By naturality of $p^!$ the class $p^!(c)\in H^{lf}_n(\ov{S'};\Q)$ can not vanish in the $r$-neighborhood of $\iota(\ov{S'})$.\\
Now for two different $\pi_1$-surjective subcomplexes carrying $c$ there is always a third one containing both. By the above it now remains to show that if $T\supset S$ is a larger $\pi_1$-surjective subcomplex carrying $c$ then we can pass from $S$ to $T$ in Definition \ref{def:fillen}. This will be shown by induction on the skeleta $T^{(k)}$ of $T$. At the start of the induction we treat the cases $k=0,1$ simultaneously. For the induction step we will need the following Lemma, which works as a substitute for \cite{BH10}*{Lemma 3.2}.

\begin{Lemma} \label{lem:attach}
Let $X$ be a connected simplicial complex and $c\in H_n(X;\Q)$ be a (simplicial) homology class. Let $Y\subset X$ be a subcomplex carrying $c$ such that $X\backslash Y$ is the disjoint union of possibly infinitely many copies of the interior of a $k$-dimensional simplex with $k\geq 2$. There is a constant $\delta_k$ such that the following holds true: If the class $c$ does not vanish in the $r$-neighborhood of $\iota(Y)$, then $c$ does not vanish in the $(\delta_k r-1)$-neighborhood of $\iota(X)$.
\end{Lemma}

\begin{proof}
Let $\Delta^k$ be the standard simplex endowed with the canonical path metric $d_k$. Consider its boundary $\p\Delta^k$ and the canonical map $(\p\Delta^{k},d_{\Delta^k})\rightarrow \p\Delta^{k}$ with Lipschitz constant $\frac{1}{\delta_k}$ for some $0<\delta_k\leq 1$. If we rescale the canonical path metric on $Y$ by $\delta_k$ then the map $(Y,d_X)\rightarrow (Y,\delta_kd_Y)$ is non expanding. To see this let $v$ and $v'$ be two points in $Y$. By definition there is a path $\gamma$ in $X$ connecting $v$ and $v'$ with $\len(\gamma)=d_X(v,v')$. By replacing all segments of $\gamma$ lying the interior of a copy of $\Delta^k$ with shortest paths connecting the endpoints in $\p\Delta^k$ we construct a path $\gamma'\subset Y$ of length $\leq \frac{1}{\delta_k}\len(\gamma)$. Hence $d_Y(v,v')\leq\frac{1}{\delta_k}d_X(v,v')$, which proves the claim. If $c$ does not vanish in the $r$-neighborhood of $\iota(Y)$ then it does not vanish in the $\delta_kr$-neighborhood of $\iota(Y,\delta_kd_Y)$. Using Lemma \ref{lem:fill}, we conclude that $c$ does not vanish in the $\delta_kr$-neighborhood of $\iota(Y,d_X)$. As the $(\delta_k r-1)$-neighborhood of $\iota(X)$ is contained in the $ \delta_kr$-neighborhood of $\iota(Y,d_X)$, this proves the Lemma.
\end{proof}

Now we can start the induction process from \cite{BH10}: First assume that $T\backslash S$ contains only one vertex $v$. Let $\ov{V}\subset\ov{T}$ be the set of lifts of $v$. For each $\ov{v}\in\ov{V}$ let $F(\ov{v})\subset\ov{S}$ be the set of vertices having a common edge with $\ov{v}$. Note that $F(\ov{v})$ is nonempty and finite since $\ov{T}$ is connected and locally finite. Let $F(\ti{v})\subset\ti{S}$ be the subset defined in an analogous fashion as $F(\overline{v})$ but with $\ov{S}$ replaced by the universal cover $\ti{S}\rightarrow S$ (and $v$ by a point $\ti{v}$ over $v$) and set $$d:={\rm diam} (F(\ti{v}))$$ measured with respect to canonical the path metric in $\ti{S}$. Then $d$ is independent of the choice of $\ti{v}$ and $r$ and furthermore $${\rm diam}F(\ov{v})\leq d.$$ Now the Lipschitz constant of the canonical map $(\ov{S},d_{\ov{T}})\rightarrow \ov{S}$ is smaller or equal $\frac{d}{2}$. As in the proof of Lemma \ref{lem:attach} we rescale the canonical path metric on $\ov{S}$ by $\frac{2}{d}$. If $p^!(c)$ does not vanish in the $r$-neighborhood of $\iota(\ov{S})$, then the same holds true for the $\frac{2r}{d}$-neighborhood of $\iota(\ov{S},\frac{2}{d}d_{\ov{S}})$. Hence $p^!(c)$ does not vanish in the $\frac{2r}{d}$-neighborhood of $\iota(\ov{S},d_{\ov{T}})$ and since the  $(\frac{2r}{d}-1)$-neighborhood of $\iota(\ov{T})$ is contained in the $\frac{2r}{d}$-neighborhood of $\iota(\ov{S},d_{\ov{T}})$ we found a lower bound for the filling radius of $p^!(c)$ in $\ov{T}$ which only depends on $S$ and $T$. If $T\backslash S$ contains more than one vertex, we apply this procedure inductively, where in each induction step we pick a vertex which has a common edge with some vertex in the subcomplex of $\ov{T}$ that has already been treated. As $T$ and $S$ are both finite this produces (if $r$ is sufficiently large) a constant $\delta_1'$ such that $p^!(c)$ does not vanish in the $\delta_1' r$-neighborhood of $\ov{S\cup T^1}$. For the induction step we assume that this holds true for the $\delta_k' r$-neighborhood of $\ov{S\cup T^k}$. By Lemma \ref{lem:attach} we conclude that $p^!(c)$ does not vanish in the $(\delta_{k+1}\delta_k' r)$-neighborhood of $\ov{S\cup T^{k+1}}$. In the end we get a constant $\delta$  only depending on $S$ and $T$ (and not on $r$), such that $c$ does not vanish in the $\delta r$-neighborhood of $\iota(\ov{T})$. Hence the notion of filling enlargeability is well defined. Next we study functorial properties of filling enlargeable homology classes (compare \cite{BH10}*{Proposition 3.4}).

\begin{Proposition} \label{prop:funct}
Let $X$ and $Y$ be connected simplicial complexes with finitely generated fundamental groups and let $\phi:X\rightarrow Y$ be a continuous map. Then following implications hold:
\begin{itemize}
\item If $\phi$ induces a surjection of fundamental groups and $\phi_*(c)$ is filling enlargeable, then $c$ is filling enlargeable.
\item If $\phi$ induces an isomorphism of fundamental groups and $c$ is filling enlargeable, then also $\phi_*(c)$ is filling enlargeable.
\end{itemize}
\end{Proposition}

\begin{proof}
First assume that $\phi_*(c)$ is filling enlargeable and $\phi$ is surjective on $\pi_1$. Let $S\subset X$ be a finite connected $\pi_1$-surjective subcomplex carrying $c$. Then $\phi(S)$ is contained in a finite $\pi_1$-surjective subcomplex $T\subset Y$ carrying $\phi_*(c)$. As $S$ and $T$ are both compact the map $\phi:S\rightarrow T$ is Lipschitz with Lipschitz constant $\frac{1}{\lambda}$. Hence, if we rescale the canonical path metric on $T$ by $\lambda$, this map is nonexpanding.\\
Let $r>0$ and choose a connected cover $p_Y:\ov{Y}\rightarrow Y$ such that $p^!(\phi_*(c))$ does not vanish in the $\frac{1}{\lambda}r$-neighborhood of $\iota(\ov{T})$. Then the same holds true for the $r$-neighborhood of $\iota(\ov{T},\lambda d_{\ov{T}})$. Let $p_X:\ov{X}\rightarrow X$ be the pullback of the covering to $X$. This will be connected since $\phi$ is surjective on $\pi_1$ and we get a map of covering spaces 
$$
\begin{CD}
\ov{S}     @>{\ov{\phi}}>>  \ov{T}\\
@V{p_X}VV        @VV{p_Y}V\\
S    @>{\phi}>>  T,
\end{CD}$$
which restricts to a bijection on each fibre. In particular it is nonexpanding, if we rescale the canonical path metric on $\ov{T}$ by $\lambda$, and by naturality it maps $p^!(c)$ to $p^!(\phi_*(c))$. Hence by Lemma \ref{lem:fill}, we see that $p^!(c)$ does not vanish in the $r$-neighborhood of $\iota(\ov{S})$ and $c$ is filling enlargeable.\\
If $\phi$ induces an isomorphism of fundamental groups, then, by the first part, we can replace $Y$ by a homotopy equivalent complex and hence we may assume that $\phi$ is an inclusion. Let $S\subset X$ be a finite $\pi_1$-surjective subcomplex carrying $c$. Then $S$ is also a subcomplex of $Y$ and it carries $\phi_*(c)$. Because $\phi$ induces an isomorphism on fundamental groups each connected cover of $X$ can be written as the restriction of a connected cover of $Y$. This shows that $\phi_*(c)$ is filling enlargeable.
\end{proof}

As a corollary we get homological invariance of filling enlargeability. Notice that, by the above, the filling enlargeable classes form a well defined subset in the group homology $H_*(\Gamma;\Q)=H_*(B\Gamma;\Q)$ of a finitely generated group $\Gamma$, since a homotopy equivalence between two different simplicial models of $B\Gamma$ identifies the subsets of filling enlargeable classes.

\begin{Corollary} \label{cor:hominv}
Let $M$ be a closed oriented manifold of dimension $n$. Then $M$ is filling enlargeable if and only if $\phi_*[M]\in H_n(B\pi_1(M);\Q)$ is filling enlargeable.\qed
\end{Corollary}

Our Theorem \ref{thm:3} follows directly from Corollary \ref{cor:hominv} and the next proposition (compare \cite{BH10}*{Theorem 3.6}).

\begin{Proposition} \label{prop:small}
Let $X$ be a connected simplicial complex with finitely generated fundamental group. Then the non filling enlargeable classes in $H_n(X;\Q)$ form a rational vector subspace.
\end{Proposition}

\begin{proof}
The class $0\in H_n(X;\Q)$ is not filling enlargeable. This follows directly from Definition \ref{def:fillen} (every finite $\pi_1$-surjective subcomplex $S$ carries $0$ but of course the 0-class vanishes in any neighbourhood of $\iota(\ov{S})$ for all $\ov{S}\rightarrow S$). Furthermore if a class is not filling enlargeable then clearly no rational multiple of it can be filling enlargeable.\\
Finally we need to show that the subset of non filling enlargeable classes is closed under addition. Let $c,d\in H_n(X;\Q)$ be non filling enlargeable and assume that $c+d$ is filling enlargeable. Then by definition there is a finite $\pi_1$-surjective subcomplex $S$ carrying $c+d$ such that for every $r>0$ there is a connected cover $\ov{X}\rightarrow X$ such that $p^!(c+d)=p^!(c)+p^!(d)$ does not vanish in the $r$-neighborhood of $\iota(\ov{S})$. But this implies that either $p^!(c)$ or $p^!(d)$ does not vanish in the $r$ neighborhood of $\ov{S}$. If we consider all natural numbers $k\geq 1$ for values of $r$ then for infinitely many $k$ either $p^!(c)$ or $p^!(d)$ does not vanish in the $k$-neighborhood of $\ov{S}$. Since $S$ carries both $c$ and $d$ we conclude that either $c$ or $d$ is filling enlargeable.
\end{proof}

\begin{Bemerkung}\label{rem:qess}
While we have already seen in Remark \ref{rem:zess} that an orientable closed width-enlargeable manifold $M^n$ is essential i.\ e. $\phi_*[M]\neq 0\in H_n(B\pi_1(M);\Z)$, it follows from Corollary \ref{cor:hominv} and Proposition \ref{prop:small} that if $M^n$ is also filling enlargeable, then it is even \emph{rationally essential} i.\ e. $\phi_*[M]\neq 0\in H_n(B\pi_1(M);\Q)$. This conclusion is strictly stronger since there are essential manifolds which are not rationally essential, for example $\R P^3$.
\end{Bemerkung}

\end{document}